\documentclass[11pt]{amsart}
\usepackage{setspace}
\usepackage{graphicx}
\usepackage{geometry}
\usepackage{multirow}
\usepackage{commath}
\usepackage{mdframed}
\usepackage{amsmath, amsthm}
\usepackage{thmtools}
\usepackage{bbm}
\usepackage{amssymb}
\usepackage{array}
\newcolumntype{M}[1]{>{\centering\arraybackslash}m{#1}}
\usepackage{mathtools}

\usepackage{amsthm}
\usepackage{float}
\usepackage{caption}
\usepackage{color}
\usepackage{enumerate}
\usepackage{tikz}
\usepackage{hyperref}
\hypersetup{
   colorlinks = true, 
   linkcolor = blue, 
   filecolor = magenta,      
   urlcolor = black, 
}
\usepackage{url}
\expandafter\def\expandafter\UrlBreaks\expandafter{\UrlBreaks
  \do\a\do\b\do\c\do\d\do\e\do\f\do\g\do\h\do\i\do\j
  \do\k\do\l\do\m\do\n\do\o\do\p\do\q\do\r\do\s\do\t
  \do\u\do\v\do\w\do\x\do\y\do\z\do\A\do\B\do\C\do\D
  \do\E\do\F\do\G\do\H\do\I\do\J\do\K\do\L\do\M\do\N
  \do\O\do\P\do\Q\do\R\do\S\do\T\do\U\do\V\do\W\do\X
  \do\Y\do\Z}
\usepackage{setspace}
\usepackage{systeme}
\usepackage{tabularx}
\usepackage{cite}
\usepackage{etoolbox}
\usepackage[british]{babel}
\newtheorem{theorem}{Theorem}
\numberwithin{theorem}{section}
\newtheorem{lemma}[theorem]{Lemma}
\newtheorem{cor}[theorem]{Corollary}

\newtheorem{fact}[theorem]{Fact}

\newtheorem{prop}[theorem]{Proposition}
\newtheorem{claim}{Claim}[theorem]
\newtheorem*{claim*}{Claim}
\newtheorem*{theorem*}{Theorem}
\newtheorem*{prop*}{Proposition}
\newtheorem*{lemma*}{Lemma}
\newtheorem*{keyobservation*}{Key Observation}
\newtheorem*{conjecture*}{Conjecture}
\newtheorem*{mainthm*}{Main Theorem}
\numberwithin{equation}{section}
\theoremstyle{definition}

\newtheorem*{notation*}{Notation}
\newtheorem{defn}[theorem]{Definition}

\newtheorem{remark}[theorem]{Remark}

\newcommand{\R}[0]{\mathbb{R}}
\newcommand{\Q}[0]{\mathbb{Q}}
\newcommand{\N}[0]{\mathbb{N}}
\newcommand{\Z}[0]{\mathbb{Z}}
\newcommand{\M}[0]{\mathcal{M}}

\def\indep{\mathrel{\raise0.2ex\hbox{\ooalign{\hidewidth$\vert$\hidewidth\cr\raise-0.9ex\hbox{$\smile$}}}}}
\newcommand*{\vrectangle}{{\ooalign{\lower.3ex\hbox{\tiny$\sqcup$}\cr\raise.4ex\hbox{\tiny$\sqcap$}}}}
\counterwithout{equation}{section}
\title{Combinatorics in 1-semi-equational theories}
\author{Mervyn Tong}
\address{Department of Pure Mathematics and Mathematical Statistics, Centre for Mathematical Sciences, Wilberforce Road, Cambridge CB3 0WB, United Kingdom}
\email{hwmt3@cam.ac.uk}
\date{\today}
\begin{document}
\begin{abstract}
    We give direct combinatorial proofs that 1-semi-equational theories satisfy two combinatorial properties that imply the non-interpretability of certain fields. Our main result is that Boolean combinations of 1-semi-equations have almost linear Zarankiewicz bounds, and hence no infinite field is interpretable in a 1-semi-equational theory; this answers a question of Chernikov--Mennen and Chernikov--Starchenko, who had respectively established almost linear Zarankiewicz bounds for Boolean combinations of $(2,1)$-semi-equations and Boolean combinations of weakly normal relations. (This was recently and independently established by Gou, Mirabi, Mittal, Tran, and Yang, using different techniques and producing different bounds.) We also show that $(k,1)$-semi-equations satisfy the $\delta$-strong Erd\H{o}s--Hajnal property with $\delta=1/6^{k-1}$; this had previously been established by Chernikov--Starchenko for an ineffective constant $\delta>0$.
\end{abstract}
\maketitle
\section{Introduction}
A key strand of model-theoretic research is the development of tools to detect the presence or absence of algebraic structure in a given first-order structure. This is most famously encapsulated by Zilber's Trichotomy Conjecture, which states that every non-trivial strongly minimal set either is locally modular --- like a module, thus having linear geometry --- or interprets an algebraically closed field. Although this was disproved by Hrushovski \cite{hrushovskiconstructions}, the idea that linearity is closely connected to the absence of a field has remained a guiding principle for model-theoretic research; versions of Zilber's Trichotomy Conjecture have been proven in various contexts, such as in o-minimal structures by Peterzil and Starchenko \cite{peterzilstarchenko}.

One approach in the literature to identifying notions of linearity is to find combinatorial properties of definable relations that preclude the interpretability of (certain) fields. One such property is that of \textit{almost linear Zarankiewicz bounds}. For a relation $R\subseteq X\times Y$ and $t,m,n\in\N^+$, write
\[z_t^R(m,n):=\max\left\lbrace |R(A,B)|: A\subseteq X, B\subseteq Y, |A|=m, |B|=n, R(A,B)\text{ is }K_{t,t}\text{-free}\right\rbrace.\]
Say that $R$ has \textit{almost linear Zarankiewicz bounds} if $z_t^R(n,n)=O_{t,\varepsilon}(n^{1+\varepsilon})$ for all $t\in\N^+$ and $\varepsilon\in\R^+$. Using the fact that the point-line incidence relations over $\Q$ and $\mathbb{F}_p^{\text{alg}}$ (for any prime $p$) do not have almost linear Zarankiewicz bounds, it is known that a theory in which all definable binary relations have almost linear Zarankiewicz bounds does not interpret an infinite field --- this follows from the proofs of \cite[Corollary 5.11]{semilinearzarankiewicz} and \cite[Corollary 6.3]{distalregularitylemma}, and a full argument is given in \cite[Proposition 4.2]{onebased}.

Almost linear Zarankiewicz bounds have been established for a variety of graphs, such as semilinear graphs over the real field \cite{semilinearzarankiewicz, semilinear1, semilinear2, semilinear3, semilinear4, semilinear5}, semilinear graphs over valued vector spaces \cite{valuedvectorspaceszarankiewicz}, Boolean combinations of $(2,1)$-semi-equations \cite{chernikovmennen}, and Boolean combinations of weakly normal relations \cite{onebased}. Of particular relevance to our discussion so far is the last item; a binary relation $R$ is \textit{weakly normal} if there is $k\in\N^+$ such that, given any $k$ fibres of $R$ with non-empty intersection, two of those fibres are equal. Among strongly minimal theories, local modularity is equivalent to a property known as \textit{one-basedness}, and it is known that a stable theory is one-based if and only if every definable binary relation is a Boolean combination of weakly normal relations \cite{hrushovskipillay}. The result in \cite{onebased}, by Chernikov and Starchenko, thus says that definable relations $R$ in stable, `linear' (one-based) theories satisfy the combinatorial notion of linearity that is almost linear Zarankiewicz bounds (in fact, $R$ has \textit{linear} Zarankiewicz bounds, namely, $z^R_t(n,n)=O_t(n)$ for all $t\in\N^+$).

In \cite{chernikovmennen}, Chernikov and Mennen propose a generalisation of one-basedness from the stable context to the NIP context: \textit{$1$-semi-equationality}. A binary relation $R$ is a \textit{$1$-semi-equation} if there is $k\in\N^+$ such that, given any $k$ distinct fibres of $R$ with non-empty intersection, two of those fibres are such that one is contained in the other; in that case, say that $R$ is a \textit{$(k,1)$-semi-equation}. A theory $T$ is \textit{1-semi-equational} (respectively, \textit{$(k,1)$-semi-equational}) if every definable binary relation is a Boolean combination of definable 1-semi-equations (respectively, $(k,1)$-semi-equations). Chernikov and Mennen prove that a binary relation is weakly normal if and only if it is a stable 1-semi-equation, and so a stable theory is one-based if and only if it is 1-semi-equational \cite[Proposition 2.19]{chernikovmennen}.

They also prove that any Boolean combination of $(2,1)$-semi-equations has almost linear Zarankiewicz bounds, and hence a $(2,1)$-semi-equational theory does not interpret an infinite field \cite[Remark 2.24]{chernikovmennen}. It is natural to ask if this result can be extended to 1-semi-equational theories; indeed, this question is posed both by Chernikov and Mennen \cite[Problem 2.26]{chernikovmennen} --- as an extension of their result on $(2,1)$-semi-equationality --- and by Chernikov and Starchenko \cite[Problem 6.2]{onebased} --- as an extension of their result on one-basedness. The main result of this paper answers this affirmatively.

\begin{theorem}[Theorem \ref{mainthm}]\label{intromainthm}
    Let $k\in \N^+$. Let $R_1, \dots, R_d, R^{(-1)}, \dots, R^{(-d')}\subseteq X\times Y$ be $(k,1)$-semi-equations, and let $R=R_1\wedge \dots\wedge R_d\wedge \neg R^{(-1)}\wedge \dots \wedge \neg R^{(-d')}$. Then, for all $t\in\N^+$,
    \[z_t^R(n,n)=\begin{cases}
        O_t(n)&\text{if }d'=d=0,\\
        O_{t,k,d}\left(n\log^{2d-2}(n+1)\right)&\text{if }d'=0\text{ and }d>0,\\
        O_{t,k,d,d'}\left(n\log^{2d+d'-1}(n+1)\right)&\text{if }d'>0.
    \end{cases}\]
\end{theorem}
The implied constants are effective: see Theorem \ref{mainthm}. Since $z_t^{E\vee E'}(n,n)\leq z_t^E(n,n)+z_t^{E'}(n,n)$ for all relations $E$ and $E'$, this shows that any Boolean combination of $(k,1)$-semi-equations has almost linear Zarankiewicz bounds. Thus, no infinite field is interpretable in a 1-semi-equational theory.

Our proof method is a direct combinatorial analysis of the partially ordered set comprising the fibres of a given Boolean combination of $(k,1)$-semi-equations. In particular, it is different to the proof of Chernikov and Mennen in \cite{chernikovmennen} for $(2,1)$-semi-equations, which used the fact that $(2,1)$-semi-equations are Boolean combinations of so-called \textit{basic} relations \cite[Definition 2.21]{chernikovmennen}, and such Boolean combinations are known to have almost linear Zarankiewicz bounds \cite{semilinearzarankiewicz}. We do not know if $(k,1)$-semi-equations are Boolean combinations of \textit{basic} relations when $k>2$. (If this were true, then every $(k,1)$-semi-equation would be a Boolean combination of $(2,1)$-semi-equations, as basic relations are $(2,1)$-semi-equations.)

Shortly before the release of this paper, almost linear Zarankiewicz bounds for Boolean combinations of 1-semi-equations were also obtained and announced by Gou, Mirabi, Mittal, Tran, and Yang, using different methods \cite{gouetal}. For the relation $R$ in Theorem \ref{intromainthm} with $d+d'>0$, they obtain the bound $z_t^R(n,n)=O_{t,k,d,d'}(n\log^{(d+d'-1)\alpha_k}(n+1))$, where $\alpha_k:=\min (k-1, 2)$.

Aside from almost linear Zarankiewicz bounds, we prove another combinatorial property of $(k,1)$-semi-equations that precludes the interpretability of certain fields: the \textit{strong Erd\H{o}s--Hajnal} property. A relation $R\subseteq X\times Y$ satisfies \textit{strong Erd\H{o}s--Hajnal (sEH)} if there is $\delta>0$ such that, for all finite $A\subseteq X$ and $B\subseteq Y$, there are $A_0\subseteq A$ and $B_0\subseteq B$ with $|A_0|\geq \delta|A|$ and $|B_0|\geq \delta|B|$ such that $R\cap (A_0\times B_0)$ is either $A_0\times B_0$ or $\emptyset$; in that case, say that $R$ satisfies \textit{$\delta$-strong Erd\H{o}s--Hajnal ($\delta$-sEH)}. Using the fact that the point-line incidence relation over $\mathbb{F}_p^{\text{alg}}$ (for any prime $p$) does not satisfy sEH, it is known that a theory in which all definable binary relations satisfy sEH does not interpret an infinite field of positive characteristic --- see \cite[proof of Corollary 6.3]{distalregularitylemma}. It may, however, interpret one of \textit{zero} characteristic: the theory of the real ordered field is distal, and hence all definable binary relations satisfy sEH \cite[Theorem 3.1]{distalregularitylemma}.

In \cite[Theorem 3.1]{21SEsEH}, Fu proves that $(2,1)$-semi-equations satisfy $\delta$-sEH with $\delta=1/64$. In \cite[Theorem 3.3]{onebased}, Chernikov and Starchenko extended this to 1-semi-equations (that is, $(k,1)$-semi-equations for any $k\in\N^+$), showing that they satisfy $\delta$-sEH for some ineffective constant $\delta>0$; their proof uses a deep result in structural graph theory from \cite{deepstructuralgraphtheory}. By a direct combinatorial analysis, we establish the following result.
\begin{theorem}[Theorem \ref{mainthmsEH}]\label{intromainthmsEH}
    For all $k\in\N^+$, if $R\subseteq X\times Y$ is a $(k,1)$-semi-equation, then $R$ satisfies $\delta$-sEH for $\delta=1/6^{k-1}$.
\end{theorem}

It is easy to check that, given relations $R_1, R_2\subseteq X\times Y$ such that $R_i$ satisfies $\delta_i$-sEH for $i\in \{1,2\}$, we have that $\neg R_1$ satisfies $\delta_1$-sEH, while $R_1\wedge R_2$ and $R_1\vee R_2$ satisfy $(\delta_1\delta_2)$-sEH. Thus, any Boolean combination of 1-semi-equations satisfies $\delta$-sEH for an effective constant $\delta>0$ given by Theorem \ref{intromainthmsEH}. (By the above, this implies that 1-semi-equational theories do not interpret a field of prime characteristic, but this point is rendered moot by the previous results.)

These results strengthen the case made in \cite{chernikovmennen} that 1-semi-equationality is an appropriate generalisation of one-basedness, and hence a candidate notion of `linearity', in the NIP context: 1-semi-equational theories satisfy combinatorial notions of linearity that were shown to hold of one-based theories in \cite{onebased}, and do not interpret an infinite field.
\subsection{Acknowledgements}
The proof of Theorem \ref{mainthmsEH} is a generalisation of an unpublished proof of Amador Martin-Pizarro and Martin Ziegler that weakly normal relations have the strong Erd\H{o}s--Hajnal property, presented by Amador in a talk at the South and East of England Model Theory Network (SEEMOD) in March 2026; I am grateful to them for allowing me to generalise their argument here. I am also grateful to Amador and Julia Wolf for our subsequent discussions as I developed and refined the proof. Finally, I thank Julia for her mentorship and for her feedback on this paper.

\textit{Soli Deo gloria.}
\subsection{Funding statement}
The author is supported by Julia Wolf's Open Fellowship from the UK Engineering and Physical Sciences Research Council (EP/Z53352X/1).
\section{Preliminaries}
In this section, we set out some key definitions and notation. For $k\in\N^+$, we write $[k]:=\{1, 2, \dots, k\}$.
\subsection{Relations and $(k,n)$-semi-equations}
Let $R\subseteq X\times Y$ be a relation. For $y\in Y$, let $R_y:=\{x\in X: (x,y)\in R\}$; this is a \textit{fibre} of $R$. For $x\in X$ and $y\in Y$, write $R(x,y)$ to mean $(x,y)\in R$. For $A\subseteq X$ and $B\subseteq Y$, write $R(A,B):=R\cap (A\times B)$; say that $R(A,B)$ is \textit{complete} if $R(A,B)=A\times B$, \textit{anti-complete} if $R(A,B)=\emptyset$, and \textit{homogeneous} if it is complete or anti-complete.
\begin{defn}    
    For $k,n\in\N^+$, say that a relation $R\subseteq X\times Y$ is a \textit{$(k,n)$-semi-equation} if, for all $y_1, \dots, y_k\in Y$ such that $\bigcap_{i\in [k]} R_{y_i}\neq\emptyset$, there are pairwise distinct $i_1, \dots, i_n, j\in [k]$ such that $\bigcap_{m\in [n]} R_{y_{i_m}}\subseteq R_{y_j}$. For $n\in\N^+$, say that $R$ is an \textit{$n$-semi-equation} if it is a $(k,n)$-semi-equation for some $k\in\N^+$.

    For $k,n\in\N^+$, say that a (first-order) theory is \textit{$(k,n)$-semi-equational} (respectively, \textit{$n$-semi-equational}) if every formula $\varphi(x;y)$ is a Boolean combination of definable $(k,n)$-semi-equations (respectively, $n$-semi-equations) $\varphi_1(x;y), \dots, \varphi_l(x;y)$.
\end{defn}
In this paper, we will mainly be concerned with 1-semi-equations. Observe that $R$ is a $(k,1)$-semi-equation if, given any $k$ distinct fibres of $R$ with non-empty intersection, two of those fibres are such that one is contained in the other. When $k=2$, this says that any two fibres of $R$ are either disjoint or comparable by inclusion, and so the set of fibres of $R$ is a \textit{laminar} set system.

We make some observations about $(k,n)$-semi-equations.
\begin{remark}\label{rmksemieqn}
    Let $k,n\in\N^+$.
    \begin{enumerate}[(i)]
        \item If $k\leq n$, then $R$ is a $(k,n)$-semi-equation if and only if $R=\emptyset$. Indeed, if $R=\emptyset$, then there are no $y_1, \dots, y_k\in Y$ such that $\bigcap_{i\in [k]} R_{y_i}\neq\emptyset$, so $R$ is vacuously a $(k,n)$-semi-equation. Conversely, if there is $y\in Y$ such that $R_y\neq\emptyset$, set $y_1=\dots=y_k=y$. Then $\bigcap_{i\in [k]} R_{y_i}\neq\emptyset$, but there are no pairwise distinct $i_1, \dots, i_n, j\in [k]$ at all since $k\leq n$.
        \item If $k>n$ and $R$ is a $(k,n)$-semi-equation, then $R$ is a $(k',n')$-semi-equation for all $k',n'\in\N^+$ such that $k\leq k'$ and $n\leq n'<k'$. Indeed, given $y_1, \dots, y_{k'}\in Y$ such that $\bigcap_{i\in [k']} R_{y_i}\neq\emptyset$, by $(k,n)$-semi-equationality, there are pairwise distinct $i_1, \dots, i_n, j\in [k]$ such that $\bigcap_{m\in [n]} R_{y_{i_m}}\subseteq R_{y_j}$. Since $n'+1\leq k'$, we can choose pairwise distinct $i_{n+1}, \dots, i_{n'}\in [k']\setminus\{i_1, \dots, i_n, j\}$, and now $\bigcap_{m\in [n']} R_{y_{i_m}}\subseteq R_{y_j}$.
        \item Boolean combinations of 1-semi-equations need not be 1-semi-equations (in fact, they need not be $(k',n')$-semi-equations for any $k',n'\in\N^+$): consider the following examples.
        \begin{itemize}
            \item \textit{Negation.} Let $X$ be an infinite set and $R_=:=\{(x,x):x\in X\}$. Then $R_=$ is a $(2,1)$-semi-equation, but $\neg R:=R_{\neq}$ is not a $(k',n')$-semi-equation for any $k',n'\in\N^+$.
            \item \textit{Disjunction.} Let $(X,<)$ be an infinite linear order, $R_<:=\{(x,y)\in X^2:x<y\}$, and $R_>:=\{(x,y)\in X^2:x>y\}$. Then $R_<$ and $R_>$ are $(2,1)$-semi-equations, but $R_<\vee R_>=R_{\neq}$ is not a $(k',n')$-semi-equation for any $k',n'\in\N^+$.
            \item \textit{Conjunction.} Let $(X,<)$ be an infinite linear order with $0\not\in X$, $R_1:=R_<\cup\{(0,x), (x,x): x\in X\}$, and $R_2:=R_>\cup\{(0,x), (x,x): x\in X\}$. Then $R_1$ and $R_2$ are $(2,1)$-semi-equations, but $R_1\wedge R_2=\{(0,x), (x,x): x\in X\}$ is not a 1-semi-equation.
        \end{itemize}
        Thus, when showing that Boolean combinations of 1-semi-equations have some particular property, \textit{a priori}, one must consider all possible Boolean combinations.
    \end{enumerate}
\end{remark}
We will encounter the following relation several times throughout this paper.
\begin{defn}
    Let $\mathbb{F}$ be a field. Let $P$ and $L$ respectively denote the set of points and lines in $\mathbb{F}^2$. The \textit{point-line incidence relation (over $\mathbb{F}$)} is the relation
    \[I:=\{(p,l)\in P\times L: p\in l\}\subseteq P\times L.\]
\end{defn}
\begin{remark}\label{pointlineis(3,2)}
    Over any field $\mathbb{F}$, the point-line incidence relation is a $(3,2)$-semi-equation: if three distinct lines $l_1, l_2, l_3$ have non-empty intersection, then this intersection is a single point $p$, and so $l_1\cap l_2=\{p\}\subseteq l_3$.
\end{remark}

We conclude this subsection by describing a subclass of 1-semi-equations.
\begin{defn}
    For $k\in\N^+$, say that a relation $R\subseteq X\times Y$ is \textit{$k$-weakly normal} if, for all $y_1, \dots, y_k\in Y$ such that $\bigcap_{i\in [k]} R_{y_i}\neq\emptyset$, there are pairwise distinct $i,j\in [k]$ such that $R_{y_i}=R_{y_j}$. Say that $R$ is \textit{weakly normal} if it is $k$-weakly normal for some $k\in\N^+$.
\end{defn}
Recall that a relation $R\subseteq X\times Y$ is \textit{stable} if there do not exist $(x_i: i\in \N)$ from $X$ and $(y_j:j\in \N)$ from $Y$ such that, for all $i,j\in\N$, $(x_i, y_j)\in R$ if and only if $i<j$.
\patchcmd{\thmhead}{(#3)}{#3}{}{}
\begin{fact}[{\cite[Proposition 2.19]{chernikovmennen}}]
    A relation $R\subseteq X\times Y$ is weakly normal if and only if it is a stable 1-semi-equation.
\end{fact}
\patchcmd{\thmhead}{#3}{(#3)}{}{}
\subsection{Partial orders}
We recall some basic definitions and facts concerning partial orders. A \textit{poset} is a partially ordered set.
\begin{defn}
    Let $(\mathcal{P},\leq)$ be a poset. Say that $P\in\mathcal{P}$ is \textit{minimal} (respectively, \textit{maximal}) if there does not exist $P'\in\mathcal{P}\setminus \{P\}$ such that $P'\leq P$ (respectively, $P\leq P'$).
    
    Say that $P_1, P_2\in \mathcal{P}$ are \textit{comparable} if $P_1\leq P_2$ or $P_2\leq P_1$, and \textit{incomparable} otherwise. A \textit{chain} (respectively, \textit{antichain}) in $\mathcal{P}$ is a subset $\mathcal{C}\subseteq \mathcal{P}$ such that any two elements of $\mathcal{C}$ are comparable (respectively, incomparable). If $\mathcal{P}$ is finite, define the \textit{width} of $\mathcal{P}$ to be the size of the largest antichain in $\mathcal{P}$.
\end{defn}
\begin{fact}[Dilworth's theorem \cite{dilworth}]
    Let $(\mathcal{P},\leq)$ be a finite poset. For $k\in\N^+$, if $\mathcal{P}$ has width $k$, then $\mathcal{P}$ is the union of (at most) $k$ chains.
\end{fact}
\subsection{Asymptotics}
Let $f(x,y),g(x,y): D\times E\to\R_{\geq 0}$, where $D,E$ are sets and $x,y$ are tuples of variables. Write $f(x,y)=O_x(g(x,y))$ and $g(x,y)=\Omega_x (f(x,y))$ if there is $C=C(x): D\to \R^+$ such that $f(x,y)\leq C g(x,y)$ for all $x\in D$ and $y\in E$.
\section{Almost linear Zarankiewicz bounds}
In this section, we prove our main result, Theorem \ref{mainthm}, that Boolean combinations of 1-semi-equations have almost linear Zarankiewicz bounds.

The \textit{Zarankiewicz problem} is the following classical problem in extremal combinatorics. Let $R(A,B)$ be a bipartite graph on vertex sets $A$ and $B$, with $|A|=|B|=n$, such that $R(A,B)$ is $K_{t,t}$-free: that is, it does not contain $K_{t,t}$, the complete $t$-by-$t$ bipartite graph, as an induced copy. Find a (good) upper bound for $|R(A,B)|$. This motivates the following definition.
\begin{defn}
    Let $t,m,n\in\N^+$. For $R\subseteq X\times Y$, we write
    \[z_t^R(m,n):=\max\left\lbrace |R(A,B)|: A\subseteq X, B\subseteq Y, |A|=m, |B|=n, R(A,B)\text{ is }K_{t,t}\text{-free}\right\rbrace.\]
    We also write
    \[z_t(m,n):=\max\left\lbrace z_t^R(m,n): R\subseteq \N\times \N\right\rbrace.\]
\end{defn}

The Zarankiewicz problem thus asks for an upper bound on $z_t(n,n)$. The best-known bound, due to K\H{o}v\'ari, S\'os, and Tur\'an \cite{kovarisosturan}, is $O_t(n^{2-1/t})$; optimality of this bound has only been established for $t\in\{2,3\}$, and is otherwise open.

We are interested in the tameness properties of those theories in which every definable binary relation $R$ is such that $z_t^R(n,n)$ is much smaller than $z_t(n,n)$. How much smaller? Any upper bound on $z_t^R(n,n)$ for all definable binary relations $R$ must be at least linear, since the equality relation $=$ is such that $z_t^=(n,n)=n$. This justifies the following definition.
\begin{defn}
    Let $R\subseteq X\times Y$. A function $f: (\N^+)^3\to\R_{\geq 0}$ is a \textit{Zarankiewicz bound} for $R$ if $z_t^R(m,n)\leq f(t,m,n)$ for all $t,m,n\in\N^+$. Say that $R$ has \textit{almost linear Zarankiewicz bounds} if, for all $t\in\N^+$ and $\varepsilon\in\R^+$,
    \[z_t^R(n,n)=O_{t,\varepsilon}(n^{1+\varepsilon}).\]
    Say that a theory has \textit{almost linear Zarankiewicz bounds} if every definable binary relation (that is, every relation defined by some formula $\phi(x;y)$ where $x, y$ are tuples of variables) does.
\end{defn}
The following fact follows from the proofs of \cite[Corollary 5.11]{semilinearzarankiewicz} and \cite[Corollary 6.3]{distalregularitylemma}, and a full argument is given in \cite[Proposition 4.2]{onebased}.
\begin{fact}\label{ALZBnofield}
    A theory with almost linear Zarankiewicz bounds does not interpret an infinite field.
\end{fact}

In \cite{chernikovmennen}, Chernikov and Mennen prove that Boolean combinations of $(2,1)$-semi-equations have almost linear Zarankiewicz bounds, and so $(2,1)$-semi-equational theories do not interpret an infinite field. In \cite{onebased}, Chernikov and Starchenko prove \textit{linear} Zarankiewicz bounds for Boolean combinations of weakly normal relations, that is, $z_t^R(n,n)=O_t(n)$ for each such Boolean combination $R$. (This implies that stable, one-based theories do not interpret an infinite field, but this was already known \cite[Corollary 4.8]{anandbook}.) 

Both sets of authors ask if the result can be extended to $(k,1)$-semi-equational theories. We answer this affirmatively. More specifically, we show the following. (Henceforth, all logarithms are natural unless otherwise specified.)
\begin{theorem}\label{mainthm}
    Let $k\in\N^+$. Let $R^{(1)}, \dots, R^{(d)}, R^{(-1)}, \dots, R^{(-d')}\subseteq X\times Y$ be $(k,1)$-semi-equations, and let $R=R^{(1)}\wedge \dots\wedge R^{(d)}\wedge \neg R^{(-1)}\wedge \dots \wedge \neg R^{(-d')}$. Then
    \[z_t^R(m,n)\leq\begin{cases}
        (t-1)(m+n)&\text{if }d'=d=0,\\
        8^{d-1}k^{4d-3}(t-1)^2(m+n)\log^{2d-2}(m+1)&\text{if }d'=0\text{ and }d>0,\\
        8^{d+d'-1}k^{4d+d'}(d'!)t^{d'+1}(m+n)\log^{2d+d'-1}(m+1)&\text{if }d'>0
    \end{cases}\]
    for all $t,m,n\in\N^+$. Thus, any Boolean combination of $(k,1)$-semi-equations has almost linear Zarankiewicz bounds, and any $(k,1)$-semi-equational theory does not interpret an infinite field.
\end{theorem}
The broad structure of our proof follows that in \cite{onebased}. Our goal is to establish almost linear Zarankiewicz bounds (of a specific form) for Boolean combinations of $(k,1)$-semi-equations; thus, by considering disjunctive normal form, it suffices to prove that:
\begin{enumerate}[(i)]
    \item if $R$ and $R'$ have almost linear Zarankiewicz bounds, so does $R\vee R'$;
    \item if $R^{(1)}, \dots, R^{(d)}$ are $(k,1)$-semi-equations, then $\neg R^{(1)}\wedge \dots \wedge \neg R^{(d)}$ has almost linear Zarankiewicz bounds;
    \item if $R$ is a $(k,1)$-semi-equation and $R'$ has almost linear Zarankiewicz bounds, then $R\wedge R'$ has almost linear Zarankiewicz bounds.
\end{enumerate}

The first of these statements is very easy to prove. The following lemma is essentially a restatement of \cite[Lemma 4.9]{onebased}, but we include an argument for completeness.

\begin{lemma}\label{disjunctions}
    For relations $R, R'\subseteq X\times Y$, 
    \[z_t^{R\vee R'}(m,n)\leq z_t^R(m,n)+z_t^{R'}(m,n)\]
    for all $t,m,n\in\N^+$.
\end{lemma}
\begin{proof}
    Let $A\subseteq X$ and $B\subseteq Y$ with $m:=|A|$ and $n:=|B|$. If $(R\vee R')(A,B)$ is $K_{t,t}$-free, then so are $R(A,B)$ and $R'(A,B)$, and hence
    \[|(R\vee R')(A,B)|\leq |R(A,B)|+|R'(A,B)|\leq z_t^R(m,n)+z_t^{R'}(m,n).\qedhere\]
\end{proof}
Let us now consider conjunctions of negated $(k,1)$-semi-equations. We begin with the empty conjunction, that is, the full relation. The following lemma is essentially a restatement of \cite[Lemma 4.7]{onebased}, but we include an argument for completeness.
\begin{lemma}\label{fullrelation}
    Let $R\subseteq X\times Y$ be the full relation (that is, $R=X\times Y$). Then
    \[z_t^R(m,n)\leq (t-1)(m+n).\]
\end{lemma}
\begin{proof}
    If $A\subseteq X$ and $B\subseteq Y$ are such that $R(A,B)$ is $K_{t,t}$-free, then either $|A|<t$ or $|B|<t$.
\end{proof}
Next, we consider a single negated $(k,1)$-semi-equation. For the purposes of proving that conjunctions of negated $(k,1)$-semi-equations have almost linear Zarankiewicz bounds, it is not necessary to consider this case separately, but this will allow us to derive better bounds (see Remark \ref{avoidonenegation}).
\begin{prop}\label{onenegzbound}
    Let $k\in\N^+$, and let $R\subseteq X\times Y$ be a $(k,1)$-semi-equation. Then
    \[z_t^{\neg R}(m,n)\leq (t-1)\left(1+t(k-1)\right)m+2(t-1)n.\]
\end{prop}
\begin{proof}
    Let $A\subseteq X$ and $B\subseteq Y$ be finite, with $m:=|A|$ and $n:=|B|$, such that $\neg R(A,B)$ is $K_{t,t}$-free. It suffices to assume that $|\neg R_b|\geq t$ for all $b\in B$ and prove that
    \[|\neg R(A,B)|\leq (t-1)\left(1+t(k-1)\right)m+(t-1)n.\]

    If $m<t$, we are done, so suppose $m\geq t$, and let $a_1, \dots, a_t\in A$ be distinct. Let $B_0:=\{b\in B: \neg R(a_s, b)\text{ for all }s\in [t]\}$. Then $\neg R(\{a_1, \dots, a_t\}, B_0)$ is complete, so $|B_0|\leq t-1$.

    For $s\in [t]$, let $B_s:=\{b\in B: R(a_s, b)\}$, so that $B=B_0\cup B_1\cup \dots \cup B_t$. We claim that $|B_s|\leq (k-1)(t-1)$ for all $s\in [t]$, which would finish the proof, as then $|B|\leq (t-1)+t(k-1)(t-1)$.

    Fix $s\in [t]$. Consider $\mathcal{P}:=\{R_b: b\in B_s\}$ as a poset (under inclusion). Each element of $\mathcal{P}$ contains $a_s$, so by $(k,1)$-semi-equationality, $\mathcal{P}$ has width at most $k-1$. By Dilworth's theorem, $\mathcal{P}$ is the union of at most $k-1$ chains. That is, we may write $B_s=B^1\cup\dots\cup B^{k-1}$ such that, for all $i\in [k-1]$, $\{R_b: b\in B^i\}$ is a chain, and hence $\{\neg R_b: b\in B^i\}$ is a chain. But then $|B^i|\leq t-1$ for all $i\in [k-1]$: if $C$ is the minimal element of $\{\neg R_b: b\in B^i\}$, then $\neg R(C,B^i)$ is complete, and $|C|\geq t$ by assumption. Thus, $|B_s|\leq (k-1)(t-1)$ as required.
\end{proof}
We now consider arbitrary conjunctions of negated $(k,1)$-semi-equations. The following two lemmas will play an important role.
\begin{lemma}\label{chainlemma}
    Let $R\subseteq A\times B$ be a relation. Let $k\in\N$, and let $B_0\subseteq B$ be such that $\{R_b: b\in B_0\}$ is a chain (under inclusion) of size at most $2^k$. Then $R(A,B_0)$ can be written as a union $F_1\cup \dots\cup F_{k+1}$, where each $F_j$ is a union of edge-disjoint boxes (that is, $A_1\times B_1\cup \dots \cup A_r\times B_r$ for some pairwise disjoint $A_1, \dots, A_r\subseteq A$ and pairwise disjoint $B_1, \dots, B_r\subseteq B$).
\end{lemma}
\begin{proof}
    It suffices to prove the lemma under the assumption that $|B_0|\leq 2^k$ and $R_b\neq R_{b'}$ for all distinct $b,b'\in B_0$. Indeed, consider the equivalence relation $\sim$ on $B$, defined by\linebreak $b\sim b':\Leftrightarrow R_b=R_{b'}$. Then $B_0/{\sim}$ satisfies this assumption, and if $R(A,B_0/{\sim})=F_1\cup\dots\cup F_{k+1}$, where each $F_j$ is a union of edge-disjoint boxes, then $R(A,B_0)=F'_1\cup\dots\cup F'_{k+1}$, where each $F'_j:=\{(a,b)\in A\times B_0: (a,[b]_\sim)\in F_j\}$ is a union of edge-disjoint boxes.

    Induct on $k$. When $k=0$, $B_0$ has at most one element. Without loss of generality, $B_0$ has exactly one element $b$, so $R(A, B_0)=R_b\times \{b\}$.

    Suppose the claim is true for $k$, and let $B_0\subseteq B$ be a subset of size $l\leq 2^{k+1}$; by the induction hypothesis, we may suppose $l>2^k$. Enumerate the elements of $B_0$ as $b_1, \dots, b_l$ such that $R_{b_1}\subseteq \dots \subseteq R_{b_l}$. Then
    \begin{align*}
        R(A,B_0)&=R_{b_{2^k}}\times \{b_{2^k}\}\cup \bigcup_{i=1}^{2^k-1}R_{b_i}\times \{b_i\}\cup \bigcup_{i=2^k+1}^l R_{b_i}\times \{b_i\}\\
        &=R_{b_{2^k}}\times \{b_{2^k}, \dots, b_l\}\cup \bigcup_{i=1}^{2^k-1}R_{b_i}\times \{b_i\}\cup \bigcup_{i=2^k+1}^l (R_{b_i}\setminus R_{b_{2^k}})\times \{b_i\}\\
        &=R_{b_{2^k}}\times \{b_{2^k}, \dots, b_l\}\cup R(A,\{b_1, \dots, b_{2^k-1}\})\cup R(A\setminus R_{b_{2^k}},\{b_{2^k+1}, \dots, b_l\}).
    \end{align*}

    By the induction hypothesis, $R(A,\{b_1, \dots, b_{2^k-1}\})$ can be written as a union $F_1\cup\dots\cup F_k$, where each $F_j$ is the union of edge-disjoint boxes. These boxes are contained in $R_{b_{2^k}}\times \{b_1, \dots, b_{2^k-1}\}$, as $R_{b_i}\subseteq R_{b_{2^k}}$ for all $i\in [2^k-1]$. Again by the induction hypothesis, $R(A\setminus R_{b_{2^k}},\{b_{2^k+1}, \dots, b_l\})$ can be written as a union $G_1\cup\dots\cup G_k$, where each $G_j$ is the union of edge-disjoint boxes contained in $(A\setminus R_{b_{2^k}})\times \{b_{2^k+1}, \dots, b_l\}$.
    
    For $j\in [k]$, let $H_j=F_j\cup G_j$, so $H_j$ is the union of edge-disjoint boxes. We conclude that
     \[R(A,B_0)=R_{b_{2^k}}\times \{b_{2^k}, \dots, b_l\}\cup H_1\cup\dots\cup H_k\]
     as required.
\end{proof}
\begin{remark}\label{remarkemptyset}
    In the statement of the lemma above, it suffices to assume that $\{R_b: b\in B_0\}\setminus\emptyset$ has size at most $2^k$. Indeed, if $B_0':=\{b\in B_0: R_b\neq\emptyset\}$, then $R(A,B_0)=R(A,B_0')$, so we may apply the lemma to $B_0'\subseteq B$ instead.
\end{remark}
Say that a function $c:(\N^+)^3\to\R_{\geq 0}$ is \textit{increasing} if, for all $(t,m,n),(t',m',n')\in(\N^+)^3\to\R^3$ with $t\leq t'$, $m\leq m'$, and $n\leq n'$, we have $c(t,m,n)\leq c(t',m',n')$.
\begin{lemma}\label{boxlemma}
    Let $R\subseteq X\times Y$ with
    \[z_t^R(m,n)\leq c(t,m,n)\cdot (m+n),\]
    where $c:(\N^+)^3\to\R_{\geq 0}$ is an increasing function. Suppose $A\subseteq X$ and $B\subseteq Y$ are finite with $|A|=m$ and $|B|=n$, $F\subseteq A\times B$ is a union of edge-disjoint boxes, and $R\cap F$ is $K_{t,t}$-free. Then
    \[|R\cap F|\leq c(t,m,n)\cdot (m+n).\]
\end{lemma}
\begin{proof}
    Say $F=A_1\times B_1\cup\dots\cup A_r\times B_r$, where $A_1, \dots, A_r\subseteq A$ are pairwise disjoint, and $B_1, \dots, B_r\subseteq B$ are pairwise disjoint. For all $j\in [r]$, $R(A_j, B_j)$ is $K_{t,t}$-free, so
    \[|R(A_j,B_j)|\leq c(t,m,n)\cdot (|A_j|+|B_j|).\]
    Summing $|R(A_j, B_j)|$ gives the desired bound.
\end{proof}
We are now ready to derive Zarankiewicz bounds for an arbitrary conjunction of negated $(k,1)$-semi-equations. By way of motivation, we give a brief description of the role played by the previous two lemmas.

Our proof is by induction on the number of conjuncts. Given a conjunction $R:=\neg R^{(1)}\wedge\dots\wedge \neg R^{(d+1)}\subseteq X\times Y$ of negated $(k,1)$-semi-equations, by induction, we have Zarankiewicz bounds for $R^{(\hat{e})}:=\bigwedge_{i\in [d+1], i\neq e} \neg R^{(i)}$, for all $e\in [d+1]$. Suppose $A\subseteq X$ and $B\subseteq Y$ are such that $R(A,B)$ is $K_{t,t}$-free. We will use Lemma \ref{chainlemma} to decompose $A\times B$ into a small number of boxes $A'\times B'$ such that, for some $e\in [d+1]$, $\neg R^{(e)}(A',B')$ is the union of a small number of sets which are unions of edge-disjoint boxes. We can then apply Lemma \ref{boxlemma} to obtain a bound on $|R(A',B')|=|R^{(\hat{e})}\cap (\neg R^{(e)}(A',B'))|$; summing over $(A',B')$ gives a bound on $|R(A,B)|$.
\begin{prop}\label{manynegszbound}
    For $d\in\N^+$, let $R^{(1)}, \dots, R^{(d)}\subseteq X\times Y$ be $(k,1)$-semi-equations, and let $R=\neg R^{(1)}\wedge \dots \wedge \neg R^{(d)}$. Then
    \[z^R_t(m,n)\leq 4^{d-1}(d!)t^{d+1}k^d (m+n)\log^{d-1} (m+1).\]
\end{prop}
\begin{proof}
    Induct on $d\in\N^+$. The case $d=1$ follows from Proposition \ref{onenegzbound}. Fix $d\in\N^+$ and suppose the claim holds for $d$. Let $R^{(1)}, \dots, R^{(d+1)}\subseteq X\times Y$ be $(k,1)$-semi-equations, and let $R=\neg R^{(1)}\wedge \dots \wedge \neg R^{(d+1)}$. For $e\in [d+1]$, let
    \[R^{(\hat{e})}:=\bigwedge_{i\in [d+1], i\neq e} \neg R^{(i)}.\]
    Let $c:=4^{d-1}(d!)t^{d+1}k^d\log^{d-1}(m+1)$. By the induction hypothesis, for all $e\in [d+1]$,
    \[z_t^{R^{(\hat{e})}}(m,n)\leq c(m+n).\]
    
    Let $A\subseteq X$ and $B\subseteq Y$ be finite, such that $R(A,B)$ is $K_{t,t}$-free. Let $m:=|A|$ and $n:=|B|$. If $m<t$, we are done, so suppose $m\geq t$, and let $a_1, \dots, a_t\in A$ be distinct. Let
    \[B_0:=\{b\in B: R(a_s, b)\text{ for all }s\in [t]\}=\{b\in B: \neg R^{(e)}(a_s, b)\text{ for all }e\in [d+1], s\in [t]\}.\] Then $R(\{a_1, \dots, a_t\}, B_0)$ is complete, so $|B_0|\leq t-1$. For $e\in [d+1]$ and $s\in [t]$, let $B_{e,s}:=\{b\in B: R^{(e)}(a_s, b)\}$, so that $B=B_0\sqcup\bigcup_{e,s}B_{e,s}$.

    Fix $e\in [d+1]$ and $s\in [t]$. Consider $\mathcal{P}:=\{R^{(e)}_b: b\in B_{e,s}\}$ as a poset (under inclusion). Each element of $\mathcal{P}$ contains $a_s$, so by $(k,1)$-semi-equationality, $\mathcal{P}$ has width at most $k-1$. By Dilworth's theorem, $\mathcal{P}$ is the union of at most $k-1$ chains. That is, we may write $B_{e,s}=B^1\cup\dots \cup B^{k-1}$ such that, for all $i\in [k-1]$, $\{R^{(e)}_b: b\in B^i\}$ is a chain, and hence $\mathcal{C}^i:=\{\neg R^{(e)}_b: b\in B^i\}$ is a chain. Since $|A|=m$, for all $i\in [k-1]$ we have $|\mathcal{C}^i\setminus \{\emptyset\}|\leq m$. Thus, by Lemma \ref{chainlemma} (and Remark \ref{remarkemptyset}), for all $i\in [k-1]$, $\neg R^{(e)}(A,B^i)$ can be written as a union $F^i_1\cup\dots\cup F^i_{\lceil \log_2 m\rceil+1}$, where each $F^i_j$ is a union of edge-disjoint boxes. For all $i\in [k-1]$, $R(A,B^i)=R^{(\hat{e})}\cap (\neg R^{(e)}(A,B^i))$ is $K_{t,t}$-free, and so
    \[|R(A,B^i)|\leq \sum_{j=1}^{\lceil \log_2 m\rceil+1}\abs{R^{(\hat{e})}\cap F^i_j}\leq (\lceil \log_2 m\rceil+1) c(m+n)\]
    by Lemma \ref{boxlemma}. Summing over $i\in [k-1]$, we have
    \[|R(A,B_{e,s})|\leq (k-1)(\lceil \log_2 m\rceil+1) c(m+n).\]
    
    Summing over $e\in [d+1]$ and $s\in [t]$, we have
    \[|R(A,B\setminus B_0)|\leq (d+1)t(k-1)(\lceil \log_2 m\rceil+1) c(m+n),\]
    and so
    \begin{align*}
        |R(A,B)|&\leq (d+1)tk(\lceil \log_2 m\rceil+1) c(m+n)\\
        &\leq 4(d+1)tk(\log(m+1)) c(m+n)\\
        &=4^d((d+1)!)t^{d+2}k^{d+1}(m+n)\log^d(m+1)
    \end{align*}
    as required.
\end{proof}
\begin{remark}\label{avoidonenegation}
    To derive a Zarankiewicz bound for $\neg R^{(1)}\wedge \dots \wedge \neg R^{(d)}$, one can avoid treating the case $d=1$ (Proposition \ref{onenegzbound}) separately by modifying the proof of Proposition \ref{manynegszbound} so that the base case of the induction is $d=0$ (Lemma \ref{fullrelation}); this would produce a Zarankiewicz bound of $O_{t,d,k}((m+n)\log^d(m+1))$.
\end{remark}
It remains to derive Zarankiewicz bounds for a conjunction $R\wedge R'$, where $R$ is a $(k,1)$-semi-equation and we are given Zarankiewicz bounds for $R'$.

We first consider the case where $R'$ is the full relation; that is, we derive Zarankiewicz bounds for a single $(k,1)$-semi-equation. For the purposes of proving that Boolean combinations of $(k,1)$-semi-equations have almost linear Zarankiewicz bounds, it is not necessary to consider this case separately, but this will allow us to derive better bounds for conjunctions of $(k,1)$-semi-equations (see Remark \ref{avoidoneequation}).
\begin{prop}\label{(k,1)zbound}
    Let $k\in\N^+$, and suppose $R\subseteq X\times Y$ is a $(k,1)$-semi-equation. Then
    \[z^R_t(m,n)\leq (k-1)(t-1)^2m+(t-1)n\]
    for all $t,m,n\in\N^+$.
\end{prop}
\begin{proof}
    Let $A\subseteq X$ and $B\subseteq Y$ be finite, with $m:=|A|$ and $n:=|B|$, such that $R(A,B)$ is $K_{t,t}$-free. It suffices to assume that $|R_b|\geq t$ for all $b\in B$ and prove that $|R(A,B)|\leq (k-1)(t-1)^2m$. Now, for distinct $b_1, \dots, b_t\in B$, we cannot have $R_{b_1}=\dots=R_{b_t}$, or else $R_{b_1}\times \{b_1, \dots, b_t\}$ would be a copy of $K_{s,t}$ in $R$, where $s=|R_{b_1}|\geq t$. Thus, it suffices to assume that $R_b\neq R_{b'}$ for all distinct $b,b'\in B$ and prove that $|R(A,B)|\leq (k-1)(t-1)m$.

    Define $B_0, B_1, \dots\subseteq B$ inductively as follows. Set $B_0:=\emptyset$. Having defined $B_0, \dots, B_j$, define $B'_j:=B\setminus (B_0\cup\dots\cup B_j)$ and $B_{j+1}:=\{b\in B'_j: R_b\text{ is maximal in the poset }\{R_{b'}: b'\in B'_j\}\}$. For all $j\in \N$, by $(k,1)$-semi-equationality, if $b_1, \dots, b_k\in B_j$ are distinct, then $\bigcap_{i\in [k]} R_{b_i}=\emptyset$: otherwise, without loss of generality, we would have $R_{b_1}\subseteq R_{b_2}$, and so $R_{b_1}\subsetneq R_{b_2}$ by assumption, contradicting the maximality of $R_{b_1}$. Thus, for all $j\in\N$, $|R(A,B_j)|\leq (k-1)m$.

    We claim that $B=B_1\cup\dots\cup B_{t-1}$. Indeed, suppose there is $b_t\in B\setminus (B_1\cup\dots \cup B_{t-1})$. Since $b_t\not\in B_{t-1}$, there is $b_{t-1}\in B_{t-1}$ such that $R_{b_t}\subseteq R_{b_{t-1}}$. Since $b_{t-1}\not\in B_{t-2}$, there is $b_{t-2}\in B_{t-2}$ such that $R_{b_{t-1}}\subseteq R_{b_{t-2}}$. Continuing, we obtain $b_t\in B_t, b_{t-1}\in B_{t-1}, \dots, b_1\in B_1$ such that $R_{b_t}\subseteq R_{b_{t-1}}\subseteq \dots\subseteq R_{b_1}$. But now we have a contradiction: $R_{b_t}\times \{b_1, \dots, b_t\}$ is a copy of $K_{s,t}$ in $R$, where $s=|R_{b_t}|\geq t$.

    We conclude that $|R(A,B)|= \sum_{j=1}^{t-1}|R(A,B_j)|\leq (k-1)(t-1)m$ as required.
\end{proof}
We now consider the case where $R'$ is an arbitrary relation with given Zarankiewicz bounds. We will need the following simple combinatorial lemma.
\begin{lemma}\label{unionbound}
    Let $k\in\N^+$. Let $X$ be a set, and let $A_1, \dots, A_k\subseteq X$ be finite sets of size at least $n$. If $A_1\cap\dots\cap A_k=\emptyset$, then $|A_1\cup\dots\cup A_k|\geq \frac{k}{k-1}n$.
\end{lemma}
\begin{proof}
    Without loss of generality, suppose $|A_k|=n$. If there is $i\in [k-1]$ such that $|A_i\cap A_k|\leq (1-\frac{1}{k-1})n$, we would be done, as then $|A_i\cup A_k|\geq |A_k|+|A_i\setminus A_k|\geq (1+\frac{1}{k-1})n$. So, suppose $|A_i\cap A_k|> (1-\frac{1}{k-1})n$ for all $i\in [k-1]$. Then $|A_k\cap A_i^\mathrm{c}|<\frac{1}{k-1}n$ for all $i\in [k-1]$, so $|A_k\setminus (A_1^\mathrm{c}\cup\dots \cup A_{k-1}^\mathrm{c})|>0$. But that precisely says that $A_1\cap\dots\cap A_k\neq \emptyset$.
\end{proof}
\begin{prop}\label{conjunctions}
    Let $k\in \N^+$. Let $R,R'\subseteq X\times Y$, such that $R$ is a $(k,1)$-semi-equation and
    \[z_t^{R'}(m,n)\leq c(t,m,n)\cdot (m+n)\]
    for all $t,m,n\in\N^+$, where $c:(\N^+)^3\to\R_{\geq 1}$ is an increasing function. Then
    \[z_t^{R\wedge R'}(m,n)\leq 8k^4(\log^2(m+1)) c(t,m,n)\cdot(m+n)\]
    for all $t,m,n\in\N^+$.
\end{prop}
Once again, Lemmas \ref{chainlemma} and \ref{boxlemma} will play an important role in the proof, analogous to their role in the proof of Proposition \ref{manynegszbound}. Given $A\subseteq X$ and $B\subseteq Y$ such that $(R\wedge R')(A,B)$ is $K_{t,t}$-free, we will use Lemma \ref{chainlemma} to decompose $A\times B$ into a small number of boxes $A'\times B'$, such that $R(A',B')$ is the union of a small number of sets which are unions of edge-disjoint boxes. We can then apply Lemma \ref{boxlemma} to obtain a bound on $|(R\wedge R')(A',B')|=|R'\cap (R(A',B'))|$; summing over $(A',B')$ gives a bound on $|(R\wedge R')(A,B)|$.
\begin{proof}[Proof of Proposition \ref{conjunctions}]
    Fix $t\in\N^+$ and finite non-empty $A\subseteq X$ and $B\subseteq Y$, such that $R\wedge R'(A,B)$ is $K_{t,t}$-free. Let $m:=|A|, n:=|B|$, and $c:=c(t,m,n)$. We claim that 
    \[|(R\wedge R')(A,B)|\leq 8k^4(\log^2(m+1))c\cdot (m+n).\]

    If $m\leq 8$ or $k=1$, we are done, so suppose $m\geq 9$ and $k\geq 2$. Let $l\in\N^+$ be minimal such that $(\frac{k}{k-1})^l>m$. For $s\in [l]$, let
    \begin{align*}
        B^s&:=\left\lbrace b\in B: \frac{m}{(\frac{k}{k-1})^s}<|R_b|\leq \frac{m}{(\frac{k}{k-1})^{s-1}}\right\rbrace,\\
        \mathcal{C}^s&:=\{R_b: b\in B^s\},\\
        \mathcal{C}^s_{\min}&:=\{C\in \mathcal{C}^s: C\text{ is minimal in }\mathcal{C}^s\text{ as a poset (under inclusion)}\}.
    \end{align*}
    Observe that $\{b\in B: R_b\neq \emptyset\}=\bigcup_{s\in [l]} B^s$. For $s\in [l]$ and $\mathcal{D}\subseteq\mathcal{C}^s_{\min}$, let
    \begin{align*}
        B^{s,\mathcal{D}}&:=\left\lbrace b\in B^s: \text{for all }C\in\mathcal{C}^s_{\min}, R_b\supseteq C\text{ if and only if }C\in\mathcal{D}\right\rbrace,\\
        A^{s,\mathcal{D}}&:=\bigcup_{b\in B^{s,\mathcal{D}}}R_b.
    \end{align*}

    We prove a series of claims.
    \begin{claim}\label{BsDdisjoint}
        For all $s\in [l]$,
        \[B^s=\bigsqcup_{\substack{\mathcal{D}\subseteq \mathcal{C}^s_{\min}\\1\leq |\mathcal{D}|\leq k-1}} B^{s,\mathcal{D}}.\]
    \end{claim}
    \begin{proof}[Proof of Claim]
        The union is trivially a disjoint union by definition. Thus, it suffices to prove that for all $b\in B^s$, $1\leq |\{C\in\mathcal{C}^s_{\min}: R_b\supseteq C\}|\leq k-1$. For the lower bound, observe that since $R_b\in\mathcal{C}^s$, it must contain some $C\in\mathcal{C}^s_{\min}$. For the upper bound, suppose we have distinct $C_1, \dots, C_k\in\mathcal{C}^s_{\min}$ such that $R_b\supseteq C_1\cup \dots\cup C_k$. By $(k,1)$-semi-equationality, $C_1\cap\dots\cap C_k=\emptyset$, so by Lemma \ref{unionbound}, $|R_b|\geq |C_1\cup\dots\cup C_k|> \frac{k}{k-1}\cdot m/(\frac{k}{k-1})^s$, which is absurd.
    \end{proof}
    \begin{claim}\label{boxbound}
        Let $s\in [l]$ and $\mathcal{D}\subseteq\mathcal{C}^s_{\min}$ with $D\neq\emptyset$. Then
        \[|(R\wedge R')(A,B^{s,\mathcal{D}})|\leq c(k-1)(\lceil\log_2 m\rceil+1)(|A^{s,\mathcal{D}}|+|B^{s,\mathcal{D}}|).\]
    \end{claim}
    \begin{proof}[Proof of Claim]
        Fix any $C_0\in\mathcal{D}$. The elements of $\{R_b: b\in B^{s,\mathcal{D}}\}$ all contain $C_0$, so it has width at most $k-1$ as a poset (under inclusion): given $k$ distinct elements, their intersection contains $C_0$ so is non-empty, and thus there must be two comparable elements by $(k,1)$-semi-equationality. By Dilworth's theorem, it is the union of at most $k-1$ chains.

        That is, we may write $B^{s,\mathcal{D}}=B_1\cup\dots\cup B_{k-1}$, such that for all $i\in [k-1]$, $\{R_b: b\in B_i\}\setminus \{\emptyset\}$ is a chain of size at most $m$ (since $|A|=m$). By Lemma \ref{chainlemma} (and Remark \ref{remarkemptyset}), for all $i\in [k-1]$, $R(A^{s,\mathcal{D}},B_i)$ can be written as a union $F_{i,1}\cup\dots\cup F_{i,(\lceil\log_2 m\rceil+1)}$, where each $F_{i,j}$ is a union of edge-disjoint boxes. Note that
        \[(R\wedge R')(A,B^{s,\mathcal{D}})=(R\wedge R')(A^{s,\mathcal{D}},B^{s,\mathcal{D}})=\bigcup_{i\in [k-1]}\bigcup_{j\in [\lceil\log_2 m\rceil+1]} R'\cap F_{i,j}.\]
        For each $i\in [k-1]$ and $j\in [\lceil\log_2 m\rceil+1]$, $|R'\cap F_{i,j}|\leq c\cdot (|A^{s,\mathcal{D}}|+|B^{s,\mathcal{D}}|)$ by Lemma \ref{boxlemma}. Summing over $i,j$ gives the desired bound.
    \end{proof}
    \begin{claim}\label{Abound}
            Let $s\in [l]$. If $\mathcal{D}_1, \dots, \mathcal{D}_k\subseteq\mathcal{C}^s_{\min}$ are distinct and of the same finite size, then $\bigcap_{i\in [k]}A^{s,\mathcal{D}_i}=\emptyset$. Therefore,
            \begin{equation}
                \sum_{\substack{\mathcal{D}\subseteq \mathcal{C}^s_{\min}\\1\leq |\mathcal{D}|\leq k-1}} \abs{A^{s,\mathcal{D}}}\leq (k-1)^2m.
            \end{equation}
        \end{claim}
        \begin{proof}[Proof of Claim]
            Suppose $\bigcap_{i\in [k]}A^{s,\mathcal{D}_i}\neq \emptyset$, so there is $b_i\in B^{s,\mathcal{D}_i}$ for each $i\in [k]$ such that $\bigcap_{i\in [k]}R_{b_i}\neq\emptyset$. By $(k,1)$-semi-equationality, without loss of generality, we have $R_{b_1}\subseteq R_{b_2}$. But then $R_{b_2}\supseteq \bigcup_{C\in\mathcal{D}_1\cup\mathcal{D}_2}C$, and $\mathcal{D}_1\cup\mathcal{D}_2\neq \mathcal{D}_2$ since $\mathcal{D}_1$ and $\mathcal{D}_2$ are distinct and of the same finite size. This contradicts the fact that $b_2\in B^{s,\mathcal{D}_2}$.

            This shows that, for all $i\in [k-1]$,
            \[\sum_{\substack{\mathcal{D}\subseteq \mathcal{C}^s_{\min}\\|\mathcal{D}|=i}} \abs{A^{s,\mathcal{D}}}\leq (k-1)m.\]
            Summing this over $i\in [k-1]$ gives the desired bound.
        \end{proof}
    Thus, for all $s\in [l]$,
    \begin{alignat*}{2}
        &|(R\wedge R')(A,B^s)|\\
        &=\sum_{\substack{\mathcal{D}\subseteq \mathcal{C}^s_{\min}\\1\leq |\mathcal{D}|\leq k-1}} \abs{(R\wedge R')(A,B^{s,\mathcal{D}})}&&\;\;\;\;\;\;\text{by Claim \ref{BsDdisjoint}}\\
        &\leq \sum_{\substack{\mathcal{D}\subseteq \mathcal{C}^s_{\min}\\1\leq |\mathcal{D}|\leq k-1}} c(k-1)(\lceil\log_2 m\rceil+1)(|A^{s,\mathcal{D}}|+|B^{s,\mathcal{D}}|)&&\;\;\;\;\;\;\text{by Claim \ref{boxbound}}\\
        &\leq c(k-1)^3(\lceil\log_2 m\rceil+1)(m+n)&&\;\;\;\;\;\;\text{by Claims \ref{BsDdisjoint} and \ref{Abound}}.
    \end{alignat*}
    Since $\{b\in B: R_b\neq \emptyset\}=\bigcup_{s\in [l]} B^s$, summing the above over $s\in [l]$, we obtain
    \begin{align*}
        |(R\wedge R')(A,B)|&\leq c(k-1)^3l(\lceil\log_2 m\rceil+1)(m+n)\\
        &\leq c(k-1)^3(\log_{k/(k-1)}m+1)(\lceil\log_2 m\rceil+1)(m+n)\\
        &\leq 4c(k-1)^3\left(\log\frac{k}{k-1}\right)^{-1}(\log^2(m+1)) (m+n)\\
        &\leq 4c(k-1)^3(2k-1)(\log^2(m+1)) (m+n)\\
        &\leq 8ck^4(\log^2(m+1)) (m+n)
    \end{align*}
where, in the penultimate line, we used the fact that $\log x\geq \frac{x-1}{x+1}$ for all $x\in\R_{\geq 1}$.
\end{proof}
\begin{cor}\label{conjunctionscor}
    For $d\in\N^+$, let $R^{(1)}, \dots, R^{(d)}\subseteq X\times Y$ be $(k,1)$-semi-equations, and let $R=R^{(1)}\wedge \dots \wedge R^{(d)}$. Then
    \[z_t^R(m,n)\leq 8^{d-1}k^{4d-3}(t-1)^2(m+n)\log^{2d-2}(m+1).\]
\end{cor}
\begin{proof}
    First apply Proposition \ref{(k,1)zbound}, and then apply Proposition \ref{conjunctions} $d-1$ times.
\end{proof}
\begin{remark}\label{avoidoneequation}
    To derive a Zarankiewicz bound for $R^{(1)}\wedge \dots \wedge R^{(d)}$, one can avoid treating the case $d=1$ (Proposition \ref{(k,1)zbound}) separately by applying Proposition \ref{conjunctions} $d$ times, starting with $R'$ as the full relation (Lemma \ref{fullrelation}); this would produce a Zarankiewicz bound of\linebreak $O_{t,d,k}((m+n)\log^{2d}(m+1))$.
\end{remark}
We put everything together to conclude our main theorem, Theorem \ref{mainthm}.
\begin{proof}[Proof of Theorem \ref{mainthm}]
    The cases where $d'=0$ are handled by Lemma \ref{fullrelation} and Corollary \ref{conjunctionscor}. When $d'>0$, first apply Proposition \ref{manynegszbound} to get a Zarankiewicz bound for $\neg R^{(-1)}\wedge \dots \wedge \neg R^{(-d')}$, and then apply Proposition \ref{conjunctions} $d$ times. The final statement now follows from Lemma \ref{disjunctions} (using disjunctive normal form) and Fact \ref{ALZBnofield}.
\end{proof}
We have shown that every Boolean combination of 1-semi-equations has almost linear Zarankiewicz bounds. Note that this fails for $n$-semi-equations when $n\geq 2$, even for a single $n$-semi-equation.
\begin{prop}
    For all $k>n\geq 2$, there is a $(k,n)$-semi-equation without almost linear Zarankiewicz bounds.
\end{prop}
\begin{proof}
    By Remark \ref{rmksemieqn}(ii), it suffices to exhibit a $(3,2)$-semi-equation without almost linear Zarankiewicz bounds. Let $I$ be the point-line incidence relation over $\mathbb{F}_p^{\text{alg}}$, for any prime $p$; by Remark \ref{pointlineis(3,2)}, $I$ is a $(3,2)$-semi-equation. It is $K_{2,2}$-free, as two distinct lines cannot intersect in two distinct points.
    
    The following construction is well-known. For $m\in\N^+$, let $P_m$ and $L_m$ respectively be the set of points and non-vertical lines in $\mathbb{F}_{p^m}^2$ (note that $\mathbb{F}_{p^m}\subseteq \mathbb{F}_p^{\text{alg}}$). The reader is invited to check that $|P_m|=|L_m|=p^{2m}$ and $|I(P_m, L_m)|=p^{3m}$, and so $z_2^I(n,n)=\Omega(n^{3/2})$.
\end{proof}
\section{Strong Erd\H{o}s--Hajnal}
In this section, we prove that (Boolean combinations of) $1$-semi-equations satisfy the \textit{strong Erd\H{o}s--Hajnal property}.
\begin{defn}\label{defnsEH}
    Let $R\subseteq X\times Y$. For $\delta>0$, say that $R$ satisfies the \textit{$\delta$-strong Erd\H{o}s--Hajnal property ($\delta$-sEH)} if, for all finite $A\subseteq X$ and $B\subseteq Y$, there are $A_0\subseteq A$ and $B_0\subseteq B$ with $|A_0|\geq \delta|A|$ and $|B_0|\geq \delta|B|$ such that $R(A_0, B_0)$ is homogeneous. Say that $R$ satisfies the \textit{strong Erd\H{o}s--Hajnal property (sEH)} if it satisfies $\delta$-sEH for some $\delta>0$.
\end{defn}
Like almost linear Zarankiewicz bounds, the strong Erd\H{o}s--Hajnal property also implies the non-interpretability of certain fields, but only those of positive characteristic. Indeed, as stated in the introduction, a theory in which every definable binary relation satisfies sEH does not interpret an infinite field of positive characteristic, but it can interpret one of zero characteristic, as shown by the theory of the real ordered field.

Unlike almost linear Zarankiewicz bounds, the strong Erd\H{o}s--Hajnal property has very good closure properties under Boolean combinations. It is easy to check that, given relations $R_1, R_2\subseteq X\times Y$ such that $R_i$ satisfies $\delta_i$-sEH for $i\in [2]$, we have that $\neg R_1$ satisfies $\delta_1$-sEH, while $R_1\wedge R_2$ and $R_1\vee R_2$ satisfy $(\delta_1\delta_2)$-sEH. Thus, although the following result is stated for $(k,1)$-semi-equations, given any relation $R$ that is a Boolean combination of $(k,1)$-semi-equations, the result can be bootstrapped to yield an effective constant $\delta>0$ for which $R$ satisfies $\delta$-sEH.
\begin{theorem}\label{mainthmsEH}
    For all $k\in\N^+$, if $R\subseteq X\times Y$ is a $(k,1)$-semi-equation, then $R$ satisfies $\delta$-sEH for $\delta=1/6^{k-1}$.
\end{theorem}
\begin{proof}
    Induct on $k$. When $k=1$, $R=\emptyset$ by Remark \ref{rmksemieqn}(i), so $R$ satisfies $1$-sEH. Suppose now $k>1$; let $R\subseteq X\times Y$ be a $(k,1)$-semi-equation, and let $A\subseteq X$, $B\subseteq Y$ be finite and non-empty.\\
    \\
    \underline{Case 1: $|\{b\in B: |R_b|\leq \frac{1}{6}|A|\}|\geq \frac{1}{2}|B|$.}

    Writing $B':=\{b\in B: |R_b|\leq \frac{1}{6}|A|\}$, we have $|B'|\geq \frac{1}{2}|B|$. Let $A':=\bigcup_{b\in B'}R_b$. If $|A'|\leq \frac{1}{2}|A|$ then we are done, as $R(A\setminus A', B')$ is anti-complete with $|A\setminus A'|\geq \frac{1}{2}|A|$. Thus, we may assume that $|A'|\geq \frac{1}{2}|A|$.

    Let $B^{\max}$ be the set of $b\in B'$ such that $R_b$ is maximal in the poset $\{R_{b'}: b'\in B'\}$ (under inclusion). Observe that $A'=\bigcup_{b\in B^{\max}}R_b$. As $|R_b|\leq \frac{1}{6}|A|\leq \frac{1}{3}|A'|$ for all $b\in B^{\max}$, there is some $\tilde{B}\subseteq B^{\max}$ such that $\tilde{A}:=\bigcup_{b\in \tilde{B}} R_b$ satisfies $|\tilde{A}|/|A'|\in [\frac{1}{3},\frac{2}{3}]$; in particular, $|\tilde{A}|\geq \frac{1}{6}|A|$. Write $B'=B_1\cup B_2$, where $B_1:=\{b\in B': R_b\subseteq \tilde{A}\}$ and $B_2:=B'\setminus B_1$. If $|B_1|\geq \frac{1}{2}|B'|$ then we are done, as $R(A'\setminus \tilde{A}, B_1)$ is anti-complete with $|A'\setminus \tilde{A}|\geq \frac{1}{3}|A'|\geq \frac{1}{6}|A|$ and $|B_1|\geq \frac{1}{2}|B'|\geq \frac{1}{4}|B|$. Thus, we may assume that $|B_2|\geq \frac{1}{2}|B'|\geq \frac{1}{4}|B|$.

    Let $\tilde{R}$ be the restriction of $R$ to $\tilde{A}\times B_2$; we claim that $\tilde{R}$ is a $(k-1,1)$-semi-equation. Indeed, suppose $b_1, \dots, b_{k-1}\in B_2$ are such that $\bigcap_{i\in [k-1]} \tilde{R}_{b_i}\neq\emptyset$, so there is $b\in \tilde{B}$ such that $R_b\cap \bigcap_{i\in [k-1]} R_{b_i}\neq\emptyset$. Since $R$ is a $(k,1)$-semi-equation, either there is $i\in [k-1]$ such that $R_{b_i}\subseteq R_b$ or $R_b\subsetneq R_{b_i}$, or there are distinct $i,j\in [k-1]$ such that $R_{b_i}\subseteq R_{b_j}$ (and so $\tilde{R}_{b_i}\subseteq \tilde{R}_{b_j}$). The latter must hold, since the former cannot: indeed, $R_{b_i}\not\subseteq R_b$ since $b_i\in B_2$, and $R_b\not\subsetneq R_{b_i}$ since $b\in \tilde{B}\subseteq B^{\max}$. This shows that $\tilde{R}$ is a $(k-1,1)$-semi-equation. Thus, by the induction hypothesis, there are $A_0\subseteq \tilde{A}$ and $B_0\subseteq B_2$ such that $\tilde{R}(A_0, B_0)=R(A_0, B_0)$ is homogeneous, with $|A_0|\geq \frac{1}{6^{k-2}}|\tilde{A}|\geq \frac{1}{6^{k-1}}|A|$ and $|B_0|\geq \frac{1}{6^{k-2}}|B_2|\geq \frac{1}{6^{k-1}}|B|$ as required.\\
    \\
    \underline{Case 2: $|\{b\in B: |R_b|\geq \frac{1}{6}|A|\}|\geq \frac{1}{2}|B|$.}

    Writing $B':=\{b\in B: |R_b|\geq \frac{1}{6}|A|\}$, we have $|B'|\geq \frac{1}{2}|B|$. Let $A':=\bigcup_{b\in B'}R_b$. If $|A'|\leq \frac{1}{2}|A|$ then we are done, as $R(A\setminus A', B')$ is anti-complete with $|A\setminus A'|\geq \frac{1}{2}|A|$. Thus, we may assume that $|A'|\geq \frac{1}{2}|A|$.
    
    Fix any $b\in B'$ such that $R_b$ is minimal in the poset $\{R_{b'}: b'\in B'\}$ (under inclusion). Then, if we define $B_1:=\{b'\in B': R_b\subseteq R_{b'}\}$ and $B_2:=\{b'\in B': R_b\not\subseteq R_{b'}\text{ and }R_{b'}\not\subseteq R_b\}$, we have that $B'=B_1\cup B_2$. If $|B_1|\geq \frac{1}{2}|B'|\geq \frac{1}{4}|B|$ then we are done, as $R(R_b, B_1)$ is complete with $|R_b|\geq \frac{1}{6}|A|$. Thus, we may assume that $|B_2|\geq \frac{1}{2}|B'|\geq \frac{1}{4}|B|$.
    
    Let $\tilde{R}$ be the restriction of $R$ to $R_b\times B_2$; we claim that $\tilde{R}$ is a $(k-1,1)$-semi-equation. Indeed, suppose $b_1, \dots, b_{k-1}\in B_2$ are such that $\bigcap_{i\in [k-1]} \tilde{R}_{b_i}\neq\emptyset$. Then, $R_b\cap \bigcap_{i\in [k-1]} R_{b_i}\neq\emptyset$. Since $R$ is a $(k,1)$-semi-equation, either there is $i\in [k-1]$ such that $R_{b_i}\subseteq R_b$ or $R_b\subseteq R_{b_i}$, or there are distinct $i,j\in [k-1]$ such that $R_{b_i}\subseteq R_{b_j}$ (and so $\tilde{R}_{b_i}\subseteq \tilde{R}_{b_j}$). By definition of $B_2$, the latter must hold, showing that $\tilde{R}$ is a $(k-1,1)$-semi-equation. Thus, by the induction hypothesis, there are $A_0\subseteq R_b$ and $B_0\subseteq B_2$ such that $\tilde{R}(A_0, B_0)=R(A_0,B_0)$ is homogeneous, with $|A_0|\geq \frac{1}{6^{k-2}}|R_b|\geq \frac{1}{6^{k-1}}|A|$ and $|B_0|\geq \frac{1}{6^{k-2}}|B_2|\geq \frac{1}{6^{k-1}}|B|$ as required.
\end{proof}
We have shown that every 1-semi-equation satisfies sEH. Note that this fails for $n$-semi-equations when $n\geq 2$.
\begin{prop}
    For all $k>n\geq 2$, there is a $(k,n)$-semi-equation that does not satisfy sEH.
\end{prop}
\begin{proof}
    By Remark \ref{rmksemieqn}(ii), it suffices to exhibit a $(3,2)$-semi-equation that does not satisfy sEH. Our example, once again, is the point-line incidence relation $I$ over $\mathbb{F}_p^{\text{alg}}$, for any prime $p$; by Remark \ref{pointlineis(3,2)}, $I$ is a $(3,2)$-semi-equation. In \cite[Proposition 6.2]{distalregularitylemma}, Chernikov and Starchenko show that $I$ does not satisfy sEH.
\end{proof}
\bibliographystyle{plainurl}
\bibliography{bib}
\end{document}